\documentclass[a4paper,11pt]{amsart}
\usepackage{amssymb,amsfonts,amsxtra,amscd,mathrsfs}
\usepackage[all]{xy}
\usepackage{fullpage}
\usepackage{graphicx}
\theoremstyle{plain}
\newtheorem{theorem}{Theorem}[section]

\newtheorem{cor}[theorem]{Corollary}

\theoremstyle{definition}
\newtheorem{defi}[theorem]{Definition}

\theoremstyle{remark}
\newtheorem{rem}[theorem]{Remark}
\numberwithin{equation}{section}
\newcommand{\mspc}{\ensuremath{\mathcal{M}_{g,n}}}

\newcommand{\dmcmp}{\ensuremath{\overline{\mathcal{M}}_{g,n}}}

\newcommand{\lcmp}{\ensuremath{\mathcal{L}\left[\dmcmp\times\mathbb{R}_{\geq 0}^n\right]}}
\newcommand{\slcmp}{\ensuremath{\mathcal{L}_{g,n}}}
\newcommand{\lcmppt}{\ensuremath{\mathcal{L}^{\circ}\left[\dmcmp\times\mathbb{R}_{\geq 0}^n\right]}}
\newcommand{\slcmppt}{\ensuremath{\mathcal{L}^{\circ}_{g,n}}}
\newcommand{\barlcmp}{\ensuremath{\mathcal{L}^\infty\left[\dmcmp\times\mathbb{R}_{\geq 0}^n\right]}}
\newcommand{\sbarlcmp}{\ensuremath{\mathcal{L}^\infty_{g,n}}}
\newcommand{\gcat}{\ensuremath{\Gamma_{\mathrm{Iso}}((g,n))}}
\newcommand{\OTFT}{\ensuremath{\mathcal{OTFT}}}
\newcommand{\Ass}{\mathcal{A}\!\mathit{ss}}
\newcommand{\gf}{\ensuremath{\mathbb{R}}}

\newcommand{\innprod}{\ensuremath{\langle -,- \rangle}}
\newcommand{\cotimes}{\ensuremath{\hat{\otimes}}}
\newcommand{\noproof}{\begin{flushright} \ensuremath{\square} \end{flushright}}
\DeclareMathOperator{\id}{id}

\DeclareMathOperator{\Aut}{Aut}

\DeclareMathOperator{\colim}{colim}
\DeclareMathOperator{\im}{Im}

\begin{document}
\title{On the extension of a TCFT to the boundary of the moduli space}
\author{Alastair Hamilton}
\address{University of Connecticut, Mathematics Department, 196 Auditorium Road, Storrs, CT 06269. USA.}
\email{hamilton@math.uconn.edu}
\begin{abstract}
The purpose of this paper is to describe an analogue of a construction of Costello in the context of finite-dimensional differential graded Frobenius algebras which produces closed forms on the decorated moduli space of Riemann surfaces. We show that this construction extends to a certain natural compactification of the moduli space which is associated to the modular closure of the associative operad, due to the absence of ultra-violet divergences in the finite-dimensional case. We demonstrate that this construction is equivalent to the ``dual construction'' of Kontsevich.
\end{abstract}
%\thanks{}
\maketitle
%\tableofcontents
%
%
%

\section{Introduction}

\subsection{Background}

Consider the open moduli space $\mspc$ of compact Riemann surfaces of genus $g$ with $n$ marked points. There is a construction, due to Costelllo \cite{costform} and known as a \emph{topological conformal field theory} (or TCFT for short), which associates to every Calabi-Yau elliptic space, a closed differential form $\omega_{\mspc}$ on this moduli space. Perhaps the simplest example of a Calabi-Yau elliptic space is the space $\Omega^\bullet(M,\mathbb{C})$ of complex-valued differential forms on a compact Riemannian manifold. More interesting examples include differential forms with values in the endomorphism bundle of a vector bundle over a compact Riemannian manifold, equipped with a metric and a flat connection.

The differential form $\omega_{\mspc}$ on this moduli space contains information of a physical character. Integrating this form over the moduli space $\mspc$  corresponds to integrating over all possible complex structures on a surface, and physically represents path integrals in which world-lines are replaced by world-sheets traced out by interacting open strings. Examples of physical theories which can be accommodated in this context include Chern-Simons theory and Yang-Mills theory, cf. \cite{costform}.

Unfortunately, it is not possible to integrate this form over the moduli space, since it is not a compactly supported form. This differential form diverges as we approach the boundary of the moduli space $\mspc$, due to the singular nature of the heat kernel at time $t=0$ from which this form is constructed. Such divergences commonly go under the heading of `ultra-violet divergences'.

The purpose of this article is to explain that this problem is due to the infinite-dimensional nature of the Calabi-Yau elliptic space. Specifically, we consider the finite-dimensional analogue of this construction and show that in this case, there is a natural compactification of the moduli space $\mspc$ to which the differential form $\omega_{\mspc}$ extends. This is due to the fact that the terms which diverge in the infinite-dimensional case, are in fact finite in the finite-dimensional case.

Notwithstanding the fact that the problematic terms are finite in the finite-dimensional case, the problem of extending this form to a compactification of the moduli space is a nontrivial problem. The compactification that we will use is a quotient of the well-known Deligne-Mumford compactification $\dmcmp$. This quotient was considered by Looijenga in \cite{looi} as a minor generalisation of the compactification used by Kontsevich \cite{kontairy} in his proof of Witten's conjectures. The issue is the complicated topological structure of the moduli space $\mspc$ and the highly stratified nature of its compactification $\dmcmp$. If we choose a compactification (such as the one-point compactification) in which the boundary is not large enough, or does not have the right combinatorial structure, extending the form will not be possible.

We will see that the combinatorics of the extension problem are solved by the combinatorial properties of an \emph{open topological field theory}, as famously axiomatised by Atiyah-Segal et al. Specifically, the terms which lie at the boundary of the moduli space are the terms of the open topological field theory associated to the Calabi-Yau elliptic space. It is well-known that one of the first consequences of the axioms of an open topological field theory is that the vector space in question must be finite-dimensional. Hence, if our Calabi-Yau elliptic space is infinite-dimensional, the terms which lie at the boundary of the moduli space, fail to exist; or to be more precise, these terms are \emph{infinite}. This accounts for why our form $\omega_{\mspc}$ diverges as we approach the boundary of moduli space.

The layout of the paper is as follows. In Section \ref{sec_forms} we recall standard material about the orbi-cell decomposition of the moduli space of curves due to Harer \cite{harer}, Mumford, Penner \cite{penner} and Thurston and the notion of orbi-cellular forms on the moduli space. In Section \ref{sec_operad} we recall the framework of modular operads due to Getzler and Kapranov \cite{getkap}, and the notion of an open topological field theory as described in this context by \cite{dualfeyn}, following the classic axioms of Atiyah-Segal et al. In Section \ref{sec_construction}, we describe Costello's construction of a closed form on $\mspc$ in the finite-dimensional context and in Section \ref{sec_extend} we show that this construction extends to the aforementioned compactification of the moduli space. In Section \ref{sec_dual}, we show that if our finite-dimensional Calabi-Yau elliptic space (which is nothing more than a differential graded Frobenius algebra in this case) is \emph{contractible}, this construction coincides with the construction introduced by Kontsevich in \cite{kontfeyn}, as formulated precisely by Chuang and Lazarev in \cite{dualfeyn}.

\subsection{Acknowledgements}

The author would like to thank Mahmoud Zeinalian for introducing him to and explaining the content of \cite{costform}. The author would also like to thank Andrey Lazarev for a number of helpful discussions.

\subsection{Outline of Costello's construction}

Let us briefly describe Costello's construction, which produces a closed form on the moduli space of curves. We start with the de Rham algebra $A:=\Omega^\bullet(M,\mathbb{C})$ of complex-valued differential forms on a compact Riemannian manifold $M$. More generally, we can start with any Calabi-Yau elliptic space, cf. \cite{costform}; but this example will be sufficient to demonstrate the general features of the construction. We consider the operator $e^{-t\Delta}:A\to A$, constructed from the Laplacian $\Delta:=[d,d_\ast]$, which describes how the states in $A$ evolve according to the heat equation. The results of \cite{heatkernels} imply that this operator is represented by a kernel
\[ K_t \in A^{\cotimes 2} \]
called the \emph{heat kernel}. Regarding $K$ as a function of $t$, the heat kernel is used to construct the closed differential form
\[ \alpha:= K+(d_\ast\cotimes\id)[K]\cdot dt\in A^{\cotimes 2}\cotimes\Omega^\bullet(\mathbb{R}_+). \]
The fact that this form is closed follows from the heat equation.

Given a compact decorated Riemann surface $R$ with negative Euler characteristic and at least one marked point, the results of Jenkins \cite{jenkins} or Strebel \cite{strebel} allow us to uniquely associate to this Riemann surface a graph $\gamma$, known as a \emph{ribbon graph}, lying embedded in the Riemann surface $R$, of which it is a deformation retract. To each edge of this graph is naturally ascribed a length $t > 0$. To each edge of this graph we attach the closed form $\alpha$ and at each vertex we integrate the product of the incoming differential forms over $M$. This yields a closed form $\omega_{\mspc}$ on $\mspc$.

\begin{figure}[htp]
\centering
\includegraphics{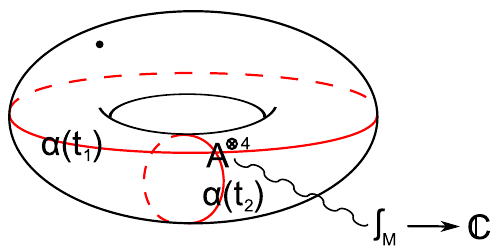}
\caption{As $t_2 \to 0$, the surface develops a nodal singularity and the integral diverges.}
\end{figure}

The problem arises as the length $t$ of a loop tends towards zero. As we shrink the length of a loop of this graph, which lies embedded in the surface $R$, this surface will develop a nodal singularity. Such nodal surfaces are precisely the types of surfaces lying at the boundary of the Deligne-Mumford compactification $\dmcmp$. Furthermore, as the length $t$ of this loop approaches zero, the heat kernel $K_t$ attached to this loop via the closed form $\alpha$ develops a divergence. This causes the associated form $\omega_{\mspc}$ to also develop a divergence as we approach the boundary of the moduli space. Contracting an edge which is not a loop poses no problems however, essentially because the diagonal in $M\times M$ on which the heat kernel blows up is a set of measure zero in $M\times M$.

\section{Moduli spaces of curves and orbi-cellular forms} \label{sec_forms}

\subsection{Ribbon graph decomposition of moduli space}

Consider the decorated moduli space $\mspc\times\mathbb{R}_{+}^n$ of Riemann surfaces of genus $g$ with $n\geq 1$  labeled marked points. The numbers in $\mathbb{R}_+^n$ are known as the \emph{perimeters} of the marked points. We will denote the moduli space with unlabeled marked points by $[\mspc\times\mathbb{R}_+^n]/\mathbb{S}_n$. It is well known from the work of Harer \cite{harer}, Mumford, Penner \cite{penner} and Thurston that this space is an orbi-cell complex. The orbi-cells are indexed by combinatorial objects called \emph{ribbon graphs} (cf. \cite{mondello} and \cite{zvonkine} for a review of this theorem).

\begin{defi}
A \emph{ribbon graph} $\gamma$ is a set, also denoted by $\gamma$, consisting of the half-edges of the graph, together with the following pieces of information:
\begin{enumerate}
\item
A partition $V(\gamma)$ of $\gamma$ corresponding to the \emph{vertices} of the graph. The cardinality of a vertex is known as its \emph{valency}, which must be $\geq 3$.
\item
A partition $E(\gamma)$ of $\gamma$ into pairs, corresponding to the \emph{edges} of $\gamma$. An edge which joins a vertex to itself is called a \emph{loop}.
\item
A cyclic ordering of the half-edges at each vertex.
\end{enumerate}
An \emph{orientation} on a ribbon graph is just an ordering of its edges modulo even permutations. There is an obvious notion of isomorphism of ribbon graphs as a mapping between the half-edges of the graph preserving the above structures. A \emph{metric ribbon graph} is just a ribbon graph together with an additional piece of data consisting of an assignment of a positive real number to each edge of the ribbon graph.
\end{defi}

A well-known construction (cf. for instance \cite{mondello} or \cite{zvonkine}) produces a (decorated) Riemann surface with unlabelled marked points from any metric ribbon graph by gluing complex strips together whose widths coincide with the lengths of the edges. This yields for every ribbon graph $\gamma$ a map
\[ \mathbb{R}_+^{E(\gamma)}\to[\mspc\times\mathbb{R}_+^{n}]/\mathbb{S}_n, \]
where $\mathbb{R}_+^{E(\gamma)}$ denotes the affine space of real-valued functions on the edges of $\gamma$. The finite group $\Aut(\gamma)$ of automorphisms of $\gamma$ acts naturally on $\mathbb{R}_+^{E(\gamma)}$ and the above map is invariant with respect to this action. The images of these maps partition the moduli space $[\mspc\times\mathbb{R}_+^n]/\mathbb{S}_n$ into \emph{orbi-cells}. This follows from the Jenkins-Strebel theory of quadratic differentials \cite{jenkins}, \cite{strebel}. These orbi-cells are indexed by isomorphism classes of ribbon graphs. The boundary of an orbi-cell $\gamma$ is found by allowing the length of an edge $e$ (not a loop), to tend to zero. The region of moduli space obtained by doing so simply coincides with the image of the orbi-cell corresponding to $\gamma/e$, the ribbon graph with the edge $e$ contracted.

Hence the decorated moduli space $[\mspc\times\mathbb{R}_+^{n}]/\mathbb{S}_n$ is an (open) orbi-cell complex whose orbi-cells are indexed by isomorphism classes of ribbon graphs. Unfortunately, it is not possible to use the above theory to obtain an orbi-cellular decomposition of the well-known Deligne-Mumford compactification $\dmcmp$ of the moduli space. This is because we need at least one marked point with positive perimeter to be able to apply the Jenkins-Strebel theory to an irreducible component of a nodal surface in $\dmcmp$. For this reason, we consider a quotient $\lcmp$ of the Deligne-Mumford compactification, in which we forget the complex structure on those irreducible components which have either no marked points, or on which the perimeters assigned to these marked points are all zero; that is to say that we remember only their \emph{topological type}. More precisely, two decorated stable curves are equivalent if, when we contract those irreducible components on which there are no marked points of positive perimeter and label the resulting nodal singularities by the number of marked points and \emph{arithmetic} genus of the contracted surfaces, the resulting curves are biholomorphic through a map which preserves the labels at the nodes; see e.g. \cite{mondello} or \cite{zvonkine} for details. It is known from the work of Kontsevich \cite{kontairy} and Looijenga \cite{looi} that this compactification also has an orbi-cellular decomposition by \emph{stable} ribbon graphs.

\begin{defi} \label{def_stabrib}
A \emph{stable ribbon graph}\footnote{These stable ribbon graphs coincide with those considered in \cite{compact}, despite their alternative description here.} $\gamma$ is a set of half-edges together with partitions of the half-edges into sets representing the vertices and edges of $\gamma$ as before, along with the following additional piece of information: rather than having a cyclic order at each vertex, we instead assign to each vertex $v\in V(\gamma)$ a compact oriented topological surface $S_v$ with nonempty boundary, together with an embedding of the incident half-edges at that vertex into the boundary of $S_v$.

\begin{figure}[htp]
\centering
\includegraphics[scale=1.3]{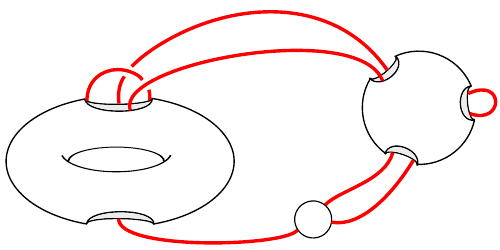}
\caption{The vertices of a stable ribbon graph are decorated by topological surfaces.}
\end{figure}

This allows us to partition the incident half-edges into \emph{cycles}, which consist of those half-edges embedded into a common boundary component, and to give a cyclic ordering of the half-edges in each cycle, coming from their embedding into the common boundary circle. If we consider only those stable ribbon graphs whose topological vertices are \emph{disks}, we recover the previous notion of a \emph{ribbon graph}.

An isomorphism of two stable ribbon graphs is a mapping between the half-edges which preserves the edge and vertex structure of the graphs; together with homeomorphisms between the topological surfaces of corresponding vertices, which must preserve the embeddings of the incident half-edges into the boundary.

As before, an orientation on a stable ribbon graph is just provided by ordering the edges. A \emph{stable metric ribbon graph} is just a stable ribbon graph with an assignment of a nonnegative real number to each edge.
\end{defi}

Again, it is possible to construct a stable curve in $[\dmcmp\times\mathbb{R}_{\geq 0}^n]/\mathbb{S}_n$ from a stable metric ribbon graph by gluing complex strips whose widths correspond to the lengths of the edges. The topological surfaces at each vertex of the stable ribbon graph are not assigned a complex structure, hence this construction takes values in the quotient $\lcmp$. Given an edge (or loop) $e$ in a stable graph $\gamma$ we can consider the stable graph $\gamma/e$, which is defined by contracting the edge $e$ in the graph. Since each of the ends of this edge are embedded in the boundary of a topological surface, we can form a new topological surface by gluing along the ends of this edge (cf. \cite{dualfeyn}).

\begin{figure}[htp]
\centering
\includegraphics{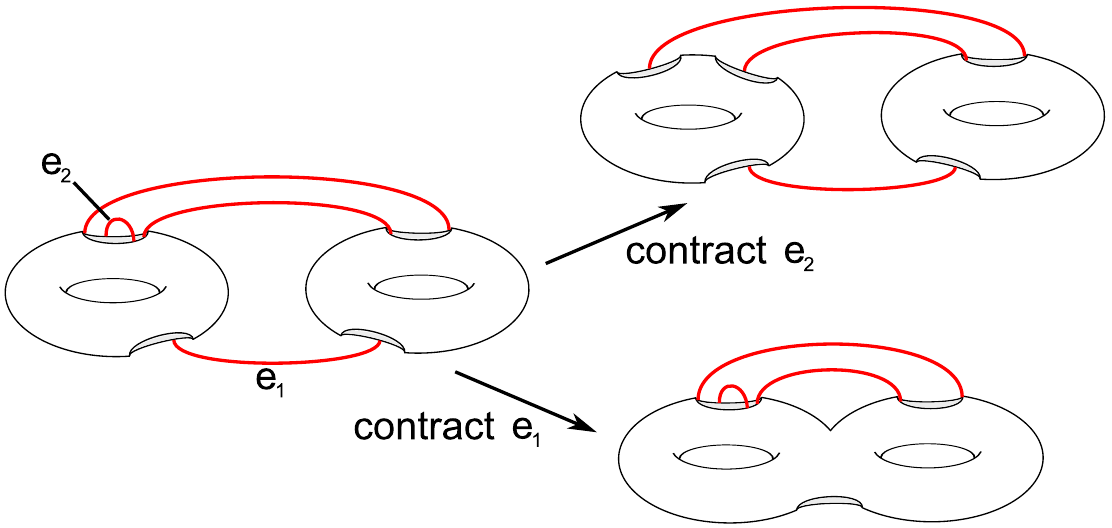}
\caption{Contracting edges and loops in a stable ribbon graph.}
\end{figure}

The construction of Riemann surfaces from stable metric ribbon graphs yields a commutative diagram:
\begin{displaymath}
\xymatrix{ \mathbb{R}_{\geq 0}^{E(\gamma)} \ar[r]^{\psi_\gamma} & \lcmp \\ \mathbb{R}_{\geq 0}^{E(\gamma/e)} \ar[ru]^{\psi_{\gamma/e}} \ar[u]^{i_{\gamma,e}} }
\end{displaymath}
where $i_{\gamma,e}$ is the canonical embedding via the zero-section.

The space $\lcmp$ is not compact. There are two natural ways to compactify this space. One is to consider the one point compactification $\lcmppt$ and the other is to consider the compactification $\barlcmp$ in which we allow the lengths of the edges of our stable ribbon graphs to become infinite. From the discussion above we have the following (cf. \cite{kontairy}, \cite{looi}, \cite{mondello}, \cite{zvonkine} \cite{dualfeyn}):

\begin{theorem}
The spaces $\barlcmp$ and $\lcmppt$ are orbi-cellular complexes.
\begin{enumerate}
\item
The orbi-cells in $\lcmppt$ are indexed by stable ribbon graphs $\gamma$ (plus a 0-cell for the point at infinity and a 0-cell for the equivalence class made up of all those surfaces with no nonvanishing perimeters), and the boundary of an orbi-cell $\gamma$ is composed of the orbi-cells $\{\gamma/e\}_{e\in E(\gamma)}$ given by contracting the edges of $\gamma$.
\item
The orbi-cells in $\barlcmp$ are indexed by stable ribbon graphs with black and white edges; the white edges of a stable ribbon graph correspond to edges of infinite length, hence those stable ribbon graphs with at least one white edge index orbi-cells lying completely on the boundary at $\infty$ of $\barlcmp$.
\end{enumerate}
\end{theorem}

\begin{rem} \label{rem_contractible}
Both the spaces $\lcmp$ and $\barlcmp$ are contractible, a contracting homotopy being provided by simply shrinking the lengths of the edges to zero. The space $\lcmppt$ is not contractible; it is clear to see that it is the suspension of the space $\mathcal{L}[\mspc\times\Delta_{n-1}]$ which has nontrivial cohomology.
\end{rem}
% Should write quotient by S_n of Looijenga?
% Check space has nontrivial cohomology?

When it is expedient to do so, we shall denote the spaces $\lcmp$, $\barlcmp$ and $\lcmppt$ by $\slcmp$, $\sbarlcmp$ and $\slcmppt$ respectively. The above discussion allows us to define the following complexes, which compute the homology of the above spaces.

\begin{defi}
\
\begin{enumerate}
\item
The complex $C_\bullet(\slcmppt)$ of orbi-cellular chains on $\slcmppt$ is linearly generated by isomorphism classes of oriented stable ribbon graphs (plus two 0-cells for the points), modulo the relation that switching the orientation of a graph changes the coefficient of that graph by a factor of $(-1)$. The differential of a graph $\gamma$ is given by the formula:
\[ \partial(\gamma):=\sum_{e\in E(\gamma)} \gamma/e \]
\item
The complex $C_\bullet(\sbarlcmp)$ of orbi-cellular chains on $\sbarlcmp$ is generated by isomorphism classes of oriented stable ribbon graphs \emph{with black and white edges}, modulo the same relation. The differential of a graph $\gamma$ is given by the formula:
\[ \partial(\gamma):=\sum_{e\in E_\mathrm{Black}(\gamma)} [\gamma/e + e\backslash\gamma], \]
where $e\backslash\gamma$ denotes the operation of replacing a black edge with a white edge and $E_\mathrm{Black}(\gamma)$ denotes the black edges of $\gamma$.
\end{enumerate}
\end{defi}

\subsection{Orbi-cellular forms}

\begin{defi}
An orbi-cellular form on the moduli space $\sbarlcmp$ is just an assignment to every (isomorphism class of) stable ribbon graph $\gamma$, a differential form $\omega_\gamma$ on $[0,\infty]^{E_{\mathrm{Black}}(\gamma)}$ satisfying the following conditions\footnote{Here $[0,\infty]$ is the one point compactification of $\mathbb{R}_{\geq 0}$. We give it the smooth structure of $[0,1]$ using the obvious map $x\mapsto \frac{1}{1-x}-1$.}:
\begin{enumerate}
\item \label{item_invariance}
$\omega_\gamma$ is invariant with respect to the natural action of $\Aut(\gamma)$
\item
These forms satisfy the gluing conditions
\begin{displaymath}
\begin{split}
i_{\gamma,e}^*[\omega_\gamma] & = \omega_{\gamma/e}, \\
i_{e,\gamma}^*[\omega_\gamma] & = \omega_{e\backslash\gamma}; \\
\end{split}
\end{displaymath}
where $i_{\gamma,e}$ denotes the zero-section as before, and $i_{e,\gamma}$ denotes the $\infty$-section.
\end{enumerate}
These forms can be differentiated in the usual way and hence they form a complex $\Omega^\bullet(\sbarlcmp)$.
\end{defi}

\begin{rem}
It is condition \eqref{item_invariance} that allows us to associate a differential form to any stable ribbon graph in an unambiguous way.
\end{rem}

\begin{rem}
Orbi-cellular forms can be defined on the spaces $\slcmp$ and $[\mspc\times\mathbb{R}^n_+]/\mathbb{S}_n$ in an analogous way (in which case only the first gluing condition is necessary).
\end{rem}

There is a natural map from the complex of orbi-cellular forms to the complex of orbi-cellular cochains given by integrating the differential forms over the corresponding orbi-cells.
\begin{equation} \label{eqn_integrate}
\int:\Omega^\bullet(\sbarlcmp) \to C^\bullet(\sbarlcmp).
\end{equation}
The following theorem is fairly standard in the absence of an action by a finite group, cf. for instance \cite{sullivan}.

\begin{theorem} \label{thm_formiso}
This map is a quasi-isomorphism and hence the complex $\Omega^\bullet(\sbarlcmp)$ computes the cohomology of $\sbarlcmp$.
\end{theorem}

\begin{proof}
Here we will try to be brief. Consider the canonical filtration $F_k\sbarlcmp$ of $\sbarlcmp$ by orbi-cells of increasing dimension. We argue by induction up this filtration using the short exact sequence
\begin{equation} \label{eqn_sesforms}
\xymatrix{ 0 \ar[r] & K^\bullet \ar[r] \ar[d]_{\int} & \Omega^\bullet(F_{k+1}) \ar[r] \ar[d]_{\int} & \Omega^\bullet(F_k) \ar[r] \ar[d]_{\int} & 0 \\ 0 \ar[r] & C^\bullet(F_{k+1}/F_k) \ar[r] & C^\bullet(F_{k+1}) \ar[r] & C^\bullet(F_k) \ar[r] & 0 }
\end{equation}
where $K^\bullet$ is the complex of orbi-cellular forms on $F_{k+1}\sbarlcmp$ which vanish on $F_k\sbarlcmp$. The complex $C^\bullet(F_{k+1}/F_k)$ likewise denotes the complex of orbi-cellular cochains on $F_{k+1}\sbarlcmp$ which vanish on $F_k\sbarlcmp$, and just consists of linear functions on $(k+1)$-orbi-cells with trivial differential. It follows from the same arguments that establish the Poincar\'e Lemma for differential forms with compact support, together with the fact that taking invariants with respect to the action of a finite group commutes with cohomology, that the induced map
\[ \int:K^\bullet \to C^\bullet(F_{k+1}/F_k) \]
is a quasi-isomorphism. Hence, if we can prove the existence of \eqref{eqn_sesforms}, Theorem \ref{thm_formiso} will follow by induction.

The only property that needs to be established is the extension property for orbi-cellular forms, namely that the top right hand map is surjective. Given a form $\omega^k\in\Omega^\bullet(F_k)$, we must extend this to a form $\omega^{k+1}\in\Omega^\bullet(F_{k+1})$ by assigning an invariant form
\[\omega^{k+1}_\gamma\in\Omega^\bullet([0,\infty]^{E_{\mathrm{Black}}(\gamma)})\]
to every $(k+1)$-orbi-cell $\gamma$, whose restriction to the boundary of $\gamma$ coincides with the form $\omega^k$:
\begin{equation} \label{eqn_extdummy}
\begin{split}
i_{\gamma,e}^*[\omega_\gamma^{k+1}] &= \omega_{\gamma/e}^k \\
i_{e,\gamma}^*[\omega_\gamma^{k+1}] &= \omega_{e\backslash\gamma}^k \\
\end{split}
\end{equation}

One proceeds by induction on the faces, as in \cite{sullivan}, to extend $\omega^k$ to a form $\omega^{k+1}_\gamma$ satisfying \eqref{eqn_extdummy} using projection maps onto opposite faces of the cube and functions which are 1 near the faces and 0 far away from them. To make this form $\Aut(\gamma)$-invariant, we symmetrise by averaging over the action of the group. That the resulting form continues to satisfy \eqref{eqn_extdummy} follows from the fact that $\omega^k$ was originally an invariant form.
\end{proof}

\begin{rem}
To prove the corresponding statements for the open cell complexes $[\mspc\times\mathbb{R}_+^n]/\mathbb{S}_n$ and $\slcmp$, one ought to replace the exact sequence \eqref{eqn_sesforms} with a Mayer-Vietoris sequence. However, since we do not actually require such a theorem, we will not endeavor to provide the details.
\end{rem}

\section{Modular operads and topological field theories} \label{sec_operad}

We begin by recalling the definition of a modular operad due to Getzler and Kapranov \cite{getkap} and then explain how this can be used to describe topological field theories.

\begin{defi}
A stable $\mathbb{S}$-module is a collection of vector spaces, or more generally, chain complexes
\[\mathcal{V}((g,n))\]
defined for $g,n\geq 0$ such that $2g+n-2>0$ and equipped with an action of $\mathbb{S}_n$ on each $\mathcal{V}((g,n))$. Morphisms of $\mathbb{S}$-modules are just equivariant maps respecting the grading by $g$ and $n$.

Given a finite set $I$ we define
\[ \mathcal{V}((g,I)):=\left[\bigoplus_{\begin{subarray}{c} \text{bijections} \\ \{1,\ldots,n\}\to I \end{subarray}} \mathcal{V}((g,n))\right]_{\mathbb{S}_n} \]
where $\mathbb{S}_n$ acts on $\mathcal{V}((g,n))$ and by permuting summands.
\end{defi}

\begin{defi}
A stable graph (with legs) is a set $G$, called the set of \emph{half-edges}, together with the following data:
\begin{enumerate}
\item
A disjoint collection of pairs of elements of $G$, denoted by $E(G)$, called the set of \emph{edges} of $G$. Those half-edges which are not part of an edge are called the \emph{legs} of $G$.
\item
A partition of $G$, denoted by $V(G)$, called the set of \emph{vertices} of $G$. We will refer to the cardinality $n(v)$ of a vertex $v\in V(G)$ as the \emph{valency} of $v$.
\item
For every vertex $v\in V(G)$, a nonnegative integer $g(v)$ called the \emph{genus} of $v$. We impose the condition that $2g(v)+n(v)-2$ must be positive at every vertex $v\in V(G)$.
\end{enumerate}
In addition, a stable graph $G$ must be \emph{connected}.
\end{defi}

\begin{rem}
Note that the definition of a \emph{stable graph} above should not be confused with the definition of a \emph{stable ribbon graph} outlined in Definition \ref{def_stabrib}. This is why we have chosen chosen to denote the former by $G$ and the latter by $\gamma$.
\end{rem}

The \emph{genus} of a stable graph is defined by the formula
\[ g(G):=\dim(H_1(G))+\sum_{v\in V(G)} g(v), \]
where $H_1(G)$ is the first homology group of the geometric realisation of $G$. The category $\gcat$ is defined to be the category whose objects are stable graphs of genus $g$ with $n$ legs, which are labelled from 1 to $n$. The morphisms are isomorphisms of stable graphs which preserve the labeling of the legs.

Given a stable graph $G$ and a stable $\mathbb{S}$-module $\mathcal{V}((g,n))$ we define
\[ \mathcal{V}((G)):= \bigotimes_{v\in V(G)}\mathcal{V}((g(v),v)). \]
There is a natural endofunctor $\mathbb{M}$ on the category of stable $\mathbb{S}$-modules defined by the formula
\[ \mathbb{M}\mathcal{V}((g,n)):=\underset{G\in\text{Iso}\Gamma((g,n))}{\colim}\mathcal{V}((G)). \]
There are natural transformations
\begin{displaymath}
\begin{split}
& \mu:\mathbb{M}\mathbb{M}\to\mathbb{M}, \\
& \eta:\id\to\mathbb{M}
\end{split}
\end{displaymath}
The map $\mu$ is given by gluing the legs of the stable graphs, located at the vertices of some parent stable graph of which they are all subgraphs, along the edges of that parent stable graph. The map $\eta$ is just the map which associates to a $\mathbb{S}$-module $\mathcal{V}((g,n))$, the corolla whose single vertex is decorated by that $\mathbb{S}$-module. These natural transformations form a triple $(\mathbb{M},\mu,\eta)$, cf. \cite{getkap}.

\begin{defi}
A modular operad is an algebra over the triple $(\mathbb{M},\mu,\eta)$. A morphism of modular operads is just a morphism of such algebras. These maps are required to commute with the differentials.
\end{defi}

\begin{defi}
Let $V$ be a finite-dimensional complex with a symmetric, even inner product $\innprod$ such that
\[ \langle d(x),y \rangle + (-1)^x\langle x,d(y) \rangle = 0. \]
The endomorphism modular operad of $V$, denoted by $\mathcal{E}[V]$, is the modular operad whose underlying $\mathbb{S}$-module is $\mathcal{E}[V]((g,n)):=(V^*)^{\otimes n}$. The structure map
\[ \mathbb{M}\mathcal{E}[V]\to\mathcal{E}[V] \]
is defined by contracting the tensors in $\mathcal{E}[V]$ along the edges of the graph using the inverse inner product $\innprod^{-1}$.
\end{defi}

\begin{rem}
This definition differs slightly from Getzler and Kapranov's in the use of $V^*$ rather than $V$, but is equivalent as the modular operads are canonically isomorphic.
\end{rem}

\begin{defi}
An \emph{algebra} over a modular operad $\mathcal{A}$ is a vector space $V$ together with a morphism of modular operads $\mathcal{A}\to\mathcal{E}[V]$.
\end{defi}

Now we want to define a modular operad whose algebras are open topological field theories (cf. \cite{dualfeyn}).

\begin{defi}
Given integers $\lambda\geq 0$ and $\nu, n \geq 1$, let $M_{\lambda,\nu,n}$ denote the category of connected compact oriented topological surfaces of genus $\lambda$ with $\nu$ boundary components and $n$ labelled intervals embedded in the boundary. That is an object in $M_{\lambda,\nu,n}$ is a connected compact oriented surface $S$ of genus $\lambda$ with $\nu$ boundary components, together with the data of $n$ orientation preserving embeddings $f_i:[0,1]\to\partial S$ for $1\leq i\leq n$.

\begin{figure}[htp]
\centering
\includegraphics{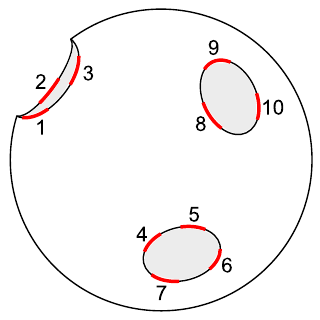}
\caption{Objects in $M_{\lambda,\nu,n}$ are surfaces with parameterised labelled intervals embedded in the boundary.}
\end{figure}

A morphism in $M_{\lambda,\nu,n}$ is just a morphism of topological spaces preserving the orientation and the labelled embedded intervals. We denote the set of isomorphism classes\footnote{This makes sense since there are only a finite number of such isomorphism classes.} by $[M_{\lambda,\nu,n}]$.
\end{defi}

\begin{defi}
The modular operad $\OTFT$ is defined as follows: its underlying $\mathbb{S}$-module is
\[ \OTFT((g,n)):=\bigoplus_{\begin{subarray}{c} \lambda\geq 0, \ \nu\geq 1: \\ 2\lambda+\nu-1=g \\  \end{subarray}} \gf[M_{\lambda,\nu,n}], \]
where $\mathbb{S}_n$ acts by relabeling the intervals embedded in the boundary. The structure map
\[ \mathbb{M} \OTFT \to \OTFT \]
of the modular operad is given by gluing surfaces along the embedded intervals on the boundary using the structure of the graph.
\end{defi}

\begin{rem}
Note that, somewhat awkwardly from a notational perspective, the genus $g$ in the modular operad $\OTFT$ does not correspond to the genus $\lambda$ of the surface, but to the quantity $2\lambda+\nu-1$, which is the dimension of the first homology group of the surface.
\end{rem}

We call an algebra  over the operad $\OTFT$ an \emph{open topological field theory}. Hence, an open topological field theory is simply a way to assign multilinear operations $V^{\otimes n}\to\gf$ to compact surfaces with boundary in such a way that these operations depend only on the topological type of the surface, and such that these operations behave in a coherent manner with respect to the possible ways in which these surfaces can be glued together. This is just a reformulation of the axioms of Atiyah-Segal et al. in terms of modular operads.

Given a modular operad $\mathcal{A}((g,n))$, we can restrict to its genus 0 part $\mathcal{A}((0,n))$. The structure map $u:\mathbb{M}\mathcal{A}\to\mathcal{A}$ of the modular operad restricts to its genus 0 part to provide $\mathcal{A}((0,n))$ with the structure of a \emph{cyclic} operad. This defines a forgetful functor from the category of modular operads to the category of cyclic operads. This functor has a left adjoint called \emph{modular closure}, which was introduced in \cite{modclose}. The modular closure of a cyclic operad $\mathcal{A}$ will be denoted by $\overline{\mathcal{A}}$. Its genus 0 part coincides with $\mathcal{A}$. The modular closure $\overline{\mathcal{A}}$ of a cyclic operad $\mathcal{A}$ which is generated by an $\mathbb{S}$-module $\mathcal{V}((n))$ modulo some set of relations is simple to construct; one simply takes the free modular operad $\mathbb{M}\mathcal{V}$ modulo the same relations. Hence, algebras over the \emph{modular} operad $\overline{\mathcal{A}}$ are canonically identified with algebras over the \emph{cyclic} operad $\mathcal{A}$.

The following is Theorem 2.7 of \cite{dualfeyn}.

\begin{theorem} \label{thm_otft}
The modular operad $\OTFT$ is canonically isomorphic to $\overline{\Ass}$, the modular closure of the cyclic associative operad.
\end{theorem}
\noproof

This leads to the following version of a well-known theorem due to Atiyah-Segal et al.

\begin{cor} \label{cor_otft}
The datum of an open topological field theory is nothing more than the datum of a differential graded Frobenius algebra.
\end{cor}

\begin{proof}
This follows from Theorem \ref{thm_otft}, since it is well-known \cite{getkap} that algebras over the cyclic associative operad are differential graded Frobenius algebras.
\end{proof}

Let us simply explain the consequences of the above theorems, rather than the details. Any open topological field theory assigns a 3-valent tensor to the unique surface $D$ of genus 0 with 3 labelled intervals embedded into its single boundary component. This 3-valent tensor is just the structure map of the Frobenius algebra of Corollary \ref{cor_otft}. Since it is clear that any surface with nonempty boundary could be constructed from copies of $D$ by gluing the intervals along the boundary together, that is to say that $D$ generates $\OTFT$, any open topological field theory is completely determined by this 3-valent tensor. Corollary \ref{cor_otft} simply says that given any differential graded Frobenius algebra, there is an open topological field theory which assigns to the surface $D$ the structure map of this Frobenius algebra.

\section{The construction} \label{sec_construction}

In this section we formulate the construction of Costello \cite{costform} producing classes on the moduli space in the context of a \emph{finite-dimensional} differential graded Frobenius algebra. Let $A$ be a finite-dimensional differential graded Frobenius algebra; that is a differential $\mathbb{Z}/2\mathbb{Z}$-graded algebra with a symmetric even invariant inner product $\innprod$:
\begin{displaymath}
\begin{split}
\langle a,bc \rangle &= \langle ab,c \rangle \\
\langle d(a),b \rangle &= -(-1)^a\langle a,d(b) \rangle
\end{split}
\end{displaymath}
Furthermore, we assume that $A$ comes equipped with an \emph{abstract Hodge decomposition}\footnote{This was referred to as a \emph{canonical} Hodge decomposition in \cite{minmod} and a \emph{harmonious} Hodge decomposition in \cite{hodge}.} (cf. \cite{minmod}), which consists of a pair of operators $s,\pi:A\to A$ satisfying:
\begin{equation} \label{eqn_hodge}
\begin{split}
ds+sd &= \id - \pi \\
s^2 &= 0 \\
\pi^2 &= \pi \\
d\pi &= \pi d =0 \\
\pi s &= s\pi=0 \\
\langle s(a),b \rangle &= (-1)^a \langle a, s(b) \rangle \\
\langle \pi(a),b \rangle &= \langle a,\pi(b) \rangle
\end{split}
\end{equation}
The above data is equivalent to a decomposition
\[ A=\im(d)\oplus\im(\pi)\oplus\im(s)\]
of $A$ into an acyclic subspace $\im(d)\oplus\im(s)$ and a subspace $\im(\pi)\cong H(A)$. Such Hodge decompositions always exist, cf. \cite{minmod}. Later, we will explain how the construction is independent of this data.

From these operators we can form a ``Laplacian'' $\Delta:=ds+sd$. Now we want to consider a deformation of the inverse inner product. We define a one-parameter family of inner products
\[K_t:=\innprod^{-1}_t\in A\otimes A \]
depending on a parameter $t\in\mathbb{R}_{\geq 0}$, such that at $t=0$ we recover our original inner product. This family is constructed to satisfy the differential equation
\begin{equation} \label{eqn_heatflow}
\frac{d}{dt}[K_t]=-(\Delta\otimes\id)[K_t]
\end{equation}
This differential equation is easily solved by the explicit formula
\[K_t:=(e^{-t\Delta}\otimes\id)\left[\innprod^{-1}\right].\]
Since $\Delta$ is idempotent, we have the following formula for $e^{-t\Delta}$:
\begin{equation} \label{eqn_exp}
e^{-t\Delta}=\pi+e^{-t}\cdot\Delta.
\end{equation}

We can consider our deformed inner product $K_t$ as an element in $A^{\otimes 2}\otimes C^\infty(\mathbb{R}_{\geq0})$. Now we will construct an inner product valued differential form $\alpha\in A^{\otimes 2}\otimes\Omega^\bullet(\mathbb{R}_{\geq 0})$ by setting
\begin{equation} \label{eqn_closedform}
\alpha:=K_t+(s\otimes\id)[K_t]\cdot dt.
\end{equation}
This form is a cycle with respect to the combination of the de Rham differential on $\mathbb{R}_{\geq 0}$ and the differential in the Frobenius algebra $A$. This follows from the ``heat equation'' \eqref{eqn_heatflow} and the fact that $\Delta$ commutes with $d$.

% Some unnecessary repetition here.
% Maybe make use of \OTFTs here too.

Let us recall how the construction works. Consider the decorated moduli space $[M_{g,n}\times\mathbb{R}_+^n]/\mathbb{S}_n$. We will define a closed orbi-cellular form on each cell of $[M_{g,n}\times\mathbb{R}_+^n]/\mathbb{S}_n$ and show that they glue together to yield a closed form on the moduli space. Suppose we take an open cell indexed by a ribbon graph $\gamma$. We use this to define a map of complexes
\begin{equation} \label{eqn_contract}
[A\otimes A\otimes \Omega^\bullet(\mathbb{R}_{\geq 0})]^{\otimes |E(\gamma)|}\to\Omega^\bullet(\mathbb{R}_{\geq 0})^{\otimes |E(\gamma)|}
\end{equation}
where $|E(\gamma)|$ is the number of edges of $\gamma$. We place one copy of $A^{\otimes 2}\otimes\Omega^\bullet(\mathbb{R}_{\geq 0})$ at each edge of the graph. At each vertex of valency $n$ we use the map $t_n:A^{\otimes n}\to\gf$
\begin{equation} \label{eqn_product}
t_n(a_1,\ldots,a_n):=\langle a_1\cdots a_{n-1},a_n \rangle,
\end{equation}
which is a map of complexes. So our desired map \eqref{eqn_contract} is
\[T:=\bigotimes_{v\in V(\gamma)}t_{|v|}\]
where $V(\gamma)$ denotes the vertices of $\gamma$.

\begin{figure}[htp]
\centering
\includegraphics[scale=1.2]{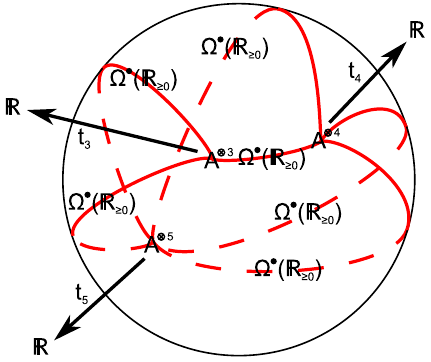}
\caption{Applying the cyclically invariant map $t_n$ to the vertices of the ribbon graph yields a differential form on the region of the moduli space $\mathcal{M}_{0,5}$ corresponding to that ribbon graph.}
\end{figure}

Consider the closed differential form
\[ \alpha^{\otimes |E(\gamma)|} \in [A^{\otimes 2}\otimes\Omega^\bullet(\mathbb{R}_{\geq 0})]^{\otimes |E(\gamma)|}.\]
We apply the map of complexes \eqref{eqn_contract} to it to produce a closed form $\omega_{\gamma}\in\Omega^\bullet(\mathbb{R}_{\geq 0})^{\otimes |E(\gamma)|}$. Using the cross product we regard it as a form $\omega_\gamma\in\Omega^\bullet(\mathbb{R}_{\geq 0}^{E(\gamma)})$, i.e. a closed form on the orbi-cell $\gamma$. The form $\omega_\gamma$ is well-defined and independent of the choices implicit in its construction. Furthermore, it is $\Aut(\gamma)$-invariant.

It remains to show that these forms $\omega_\gamma$ glue together to yield a form on the decorated moduli space. We want to show that $\omega_{\gamma/e}$ is the pull-back $i_{\gamma,e}^*[\omega_\gamma]$. The pull-back of the differential form $\alpha\in A^{\otimes 2}\otimes\Omega^\bullet(\mathbb{R}_{\geq 0})$ to $t=0$ is just the tensor $K_0:=\innprod^{-1}$. Consequently, the differential form $i_{\gamma,e}^*[\omega_\gamma]$ is obtained by replacing the differential form $\alpha$ on the shrinking edge $e$ (not loop) with the inverse inner product and applying the map $T$. One can check that joining the tensors $t_n$ and $t_m$ with the inverse inner product $\innprod^{-1}$ yields the tensor $t_{n+m-2}$, which is precisely the map which is placed at the newly formed vertex in $\gamma/e$ which results from collapsing the edge $e$, hence these forms glue to a global form on the moduli space.

\section{Extension to the compactification} \label{sec_extend}

In this section we show that the construction outlined in the previous section extends to the compactification $\slcmp:=\lcmp$ of moduli space defined in Section \ref{sec_forms}. As was mentioned in Remark \ref{rem_contractible}, this compactification is in fact contractible, hence Costello's construction produces something trivial in the case of a finite-dimensional algebra (except of course for the class in $H^0([\mspc\times\mathbb{R}_+^n]/\mathbb{S}_n)\cong\gf$, which we will see corresponds to the amplitude of the associated open topological field theory). This extension is not possible in the infinite-dimensional case due to the presence of ultra-violet divergences in the integrands, hence the classes obtained in this case may be nontrivial. Nonetheless, we show that if our finite-dimensional differential graded Frobenius algebra is \emph{contractible}, this form is the pull-back of a cocycle on a noncontractible space. Hence, if our finite-dimensional algebra is contractible, we may hope to be able to use it to produce nontrivial classes.

Let us now show that the closed form we constructed on $[\mspc\times\mathbb{R}_+^n]/\mathbb{S}_n$ in Section \ref{sec_construction} extends to the partial compactification $\slcmp$. To each open orbi-cell indexed by a stable ribbon graph $\gamma$, we must assign a differential form $\omega_\gamma$. We assign, as before, the differential form $\alpha\in A^{\otimes 2}\otimes\Omega^\bullet(\mathbb{R}_{\geq 0})$ to each edge of $\gamma$. To each vertex $v\in V(\gamma)$ we assign a map of complexes
\[t_v:A^{\otimes|v|}\to \gf\]
determined as follows. The vertex $v$ is decorated by a topological surface, with the incident half-edges of $v$ embedded in the boundary of this surface. Since $A$ is a differential graded Frobenius algebra it determines, by Corollary \ref{cor_otft}, an algebra over the modular operad $\OTFT$. This open topological field theory assigns to this topological surface the desired map $t_v:A^{\otimes|v|}\to\gf$. Note that when this topological surface is a disk, this map agrees with the map \eqref{eqn_product} defined in Section \ref{sec_construction}, hence this construction will extend the construction described in the previous section.

\begin{figure}[htp] \label{fig_otft}
\centering
\includegraphics[scale=1.2]{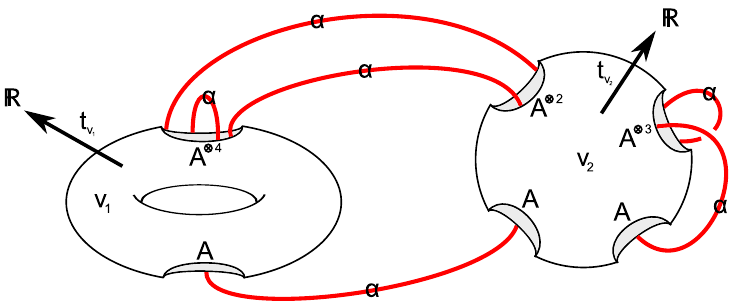}
\caption{At each vertex of our stable ribbon graph, we contract the incoming tensors using the open topological field theory determined by the Frobenius algebra $A$.}
\end{figure}

Tensoring these maps together yields a map of complexes
\begin{equation} \label{eqn_otft}
T:=\bigotimes_{v\in V(\gamma)} t_v: A^{\otimes 2|E(\gamma)|}\to\gf.
\end{equation}
Hence, when this map is applied to the differential form
\[ \alpha^{\otimes |E(\gamma)|}\in[A^{\otimes 2}\otimes \Omega^\bullet(\mathbb{R}_{\geq 0})]^{\otimes |E(\gamma)|}\]
defined by Equation \eqref{eqn_closedform}, we get a closed form $\omega_\gamma\in\Omega^\bullet(\mathbb{R}^{E(\gamma)}_{\geq 0})$ on this orbi-cell. It follows from the axioms of an open topological field theory, as expressed in Section \ref{sec_operad} through the language of modular operads, that the form $\omega_\gamma$ is $\Aut(\gamma)$-invariant and does not depend upon the various choices implicit in its construction.

It remains to check that these forms glue to yield a form on the space $\slcmp$. To do this we must show that
\[ i_{\gamma,e}^*[\omega_\gamma] = \omega_{\gamma/e}. \]
The pullback $i_{\gamma,e}^*[\omega_\gamma]$ of $\omega_\gamma$ is obtained by replacing the differential form $\alpha$ on the edge (or loop) $e$, by its value at $t=0$, the inverse inner product $\innprod^{-1}$. Hence we contract the corresponding tensors in the map $T:A^{\otimes 2|E(\gamma)|}\to \gf$ using $\innprod^{-1}$, and apply the result to the form $\alpha^{\otimes |E(\gamma/e)|}$ to arrive at the value for the pullback $i_{\gamma,e}^*[\omega_\gamma]$.

\begin{figure}[htp]
\centering
\includegraphics{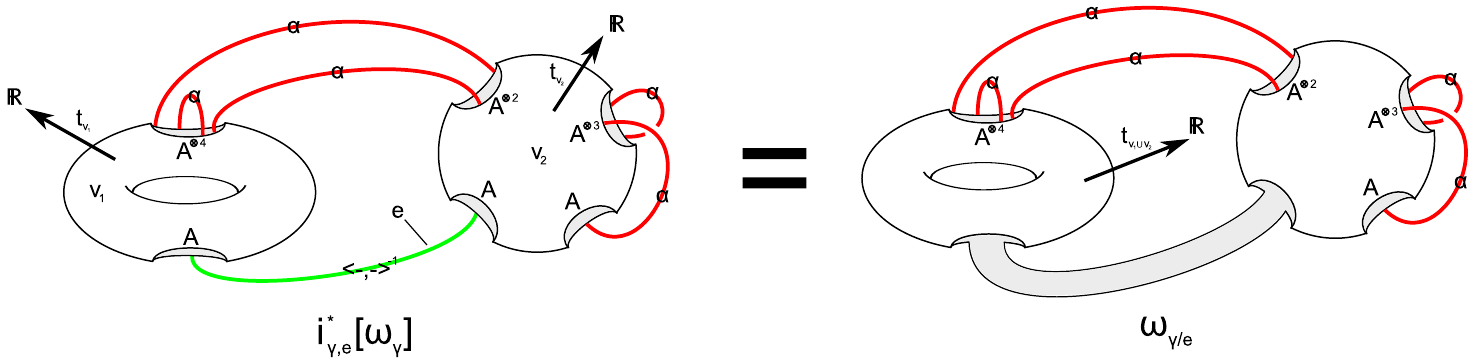}
\caption{The axioms of an open topological field theory ensure that the forms $\omega_\gamma$ glue to yield a form on the compactification $\slcmp$.}
\end{figure}

Since an open topological field theory is an algebra over the modular operad $\OTFT$, it follows from these axioms that the tensor that we get from $T:A^{\otimes 2|E(\gamma)|}\to \gf$ by contracting using $\innprod^{-1}$, is the same as the tensor that is obtained from the collection of open topological field theories when we replace the edge $e$ by a thin strip and glue the ends of this strip to the corresponding topological surface(s). By definition, this is exactly the tensor that is assigned by \eqref{eqn_otft} to $\gamma/e$. Hence we see that the form $i_{\gamma,e}^*[\omega_\gamma]$ coincides with the form $\omega_{\gamma/e}$ that is assigned to the orbi-cell $\gamma/e$, and thus the forms $\omega_\gamma$ glue to yield an orbi-cellular form $\omega$ on $\slcmp$.

From this description we may understand why the problem of extending a topological conformal field theory to the boundary of the moduli space arises when the dimension of the algebra $A$ is \emph{infinite}. As we head towards the boundary of the moduli space, the complex structure on our Riemann surface degenerates. At the boundary of the moduli space we are left with nodal surfaces, some of whose irreducible components retain only their topological structure. To those components with complex structure, we may assign a conformal field theory using the heat kernel etc., as described by Costello in \cite{costform}, but to those components with only topological structure, we may assign only a topological field theory.

It is well-known that one of the immediate consequences of the axioms of a topological field theory is that the underlying space must be finite-dimensional. More accurately, when the space is infinite-dimensional, the terms of our topological field theory are infinite. Consider, for example, the number assigned by an open topological field theory to the annulus. This number is just the trace of the identity operator, i.e. the dimension of the space. Hence, when the dimension of the space is infinite, this number will be infinite, and in general the terms lying at the boundary of the moduli space contributed by the open topological field theory will be divergent. This problem does not arise with the conformal field theory because the operator $e^{-t\Delta}$ is of trace class for $t>0$, its trace being given by integrating the heat kernel over the diagonal. It is only at $t=0$, when the conformal structure on the Riemann surface degenerates, that its trace becomes infinite.

Now we consider the full compactifications $\sbarlcmp$ and $\slcmppt$ of the decorated moduli space $[\mspc\times\mathbb{R}_+^n]/\mathbb{S}_n$, rather than the partial compactification $\slcmp$. There is a map
\begin{equation} \label{eqn_collapse}
\barlcmp\to\lcmppt
\end{equation}
which collapses the boundary at infinity to the point. That is to say that it contracts all the orbi-cells with at least one white edge.

The differential form $\omega$ on $\slcmp$ that we have just constructed extends to the compactification $\sbarlcmp$. This follows from equation \eqref{eqn_exp}, from which we see that
\begin{equation} \label{eqn_limit}
\lim_{t\to\infty}\left[K_t\right]=(\pi\otimes\id)\left[\innprod^{-1}\right]\in A^{\otimes 2}.
\end{equation}
To describe the extension of the form $\omega$ to $\sbarlcmp$, it suffices to specify the forms $\omega_\gamma$ which are assigned to the orbi-cells $\gamma$ in $\sbarlcmp$ with at least one white edge. On the black edges of $\gamma$ is placed the form $\alpha\in A^{\otimes 2}\otimes\Omega^\bullet(\mathbb{R}_{\geq 0})$, as before, and on the white edges of $\gamma$ is placed the tensor \eqref{eqn_limit}, which is a cycle with respect to $d$. Applying the map $T:A^{\otimes 2|E(\gamma)|}\to\gf$ as before yields a closed form $\omega_\gamma\in\Omega^\bullet(\mathbb{R}^{E_{\mathrm{Black}}(\gamma)}_{\geq 0})$ on the orbi-cell $\gamma$. It follows from equation \eqref{eqn_limit} that these forms satisfy the necessary gluing property.

Now suppose that our differential graded Frobenius algebra is \emph{contractible}. Then we claim that the closed differential form $\omega$ on $\sbarlcmp$ is the pull-back of a cocycle on the one point compactification $\slcmppt$. Since $A$ is contractible, we have $\pi=0$ and the differential form $\omega$ vanishes as we approach the point at infinity. Now, we would like to say that this implies that $\omega$ extends to a form on $\slcmppt$, but clearly the map \eqref{eqn_collapse} which shrinks the boundaries of cells to a point is somewhat awkward to deal with in the context of smooth structure, so we will work with the orbi-cellular cochains instead. Since $\lim_{t\to\infty} K_t=0$, it clearly follows that there is an orbi-cellular cochain in $C^\bullet(\slcmppt)$ (which must be unique and a cocycle) whose pullback to $\sbarlcmp$ via the map \eqref{eqn_collapse} yields the orbi-cellular cocycle associated to the form $\omega$ on $\sbarlcmp$ by \eqref{eqn_integrate}.

\section{Dual construction} \label{sec_dual}

In this section we show the equivalence of Costello's construction, in the finite-dimensional case, to the ``dual construction'' of Kontsevich \cite{kontfeyn}.

The key formula is the following:
\begin{equation} \label{eqn_propagator}
\begin{split}
\int_0^\infty (s\otimes\id)[K_t] dt & = \int_0^\infty (se^{-t\Delta}\otimes\id) \left[\innprod^{-1}\right] dt \\
&= \int_0^\infty ([s\pi+e^{-t}\cdot s\Delta]\otimes\id)\left[\innprod^{-1}\right] dt \\
&= \left(\int_0^\infty e^{-t} dt\right)(s\Delta\otimes\id)\left[\innprod^{-1}\right] \\
&= (s\otimes\id)\left[\innprod^{-1}\right].
\end{split}
\end{equation}
The second line follows from Equation \eqref{eqn_exp}. The third and fourth lines follow from the identities \eqref{eqn_hodge} for the Hodge decomposition.

From this we can show Costello's construction agrees with Kontsevich's as follows. Given a differential graded Frobenius algebra $A$ with a Hodge decomposition, consider the corresponding closed orbi-cellular form $\omega$ on $\sbarlcmp$ defined in Section \ref{sec_extend}. By \eqref{eqn_integrate}, this orbi-cellular form gives rise to an orbi-cellular cocycle on $\sbarlcmp$ by integrating $\omega$ over the orbi-cells of $\sbarlcmp$. If the differential graded Frobenius algebra is \emph{contractible}, it was shown at the end of the previous section that this cocycle lifts via the map \eqref{eqn_collapse} to a cocycle on $\slcmppt$.

This orbi-cellular cocycle associates to every stable ribbon graph $\gamma\in C_\bullet(\slcmppt)$, the following number. One takes the map \eqref{eqn_otft}
\[ T:A^{\otimes 2|E(\gamma)|}\to\gf \]
and applies it to the form
\[ \alpha^{\otimes |E(\gamma)|}\in [A^{\otimes 2}\otimes\Omega^\bullet(\mathbb{R}_{\geq 0})]^{\otimes|E(\gamma)|} \]
defined in Equation \eqref{eqn_closedform}, to yield a differential form $T(\alpha)\in\Omega^\bullet(\mathbb{R}^{E(\gamma)}_{\geq 0})$, as in Figure \eqref{fig_otft}. One then integrates the differential form $T(\alpha)$ to produce the number assigned to this orbi-cell.

Equivalently, we can first integrate $\alpha^{\otimes |E(\gamma)|}$ to produce a tensor
\[ \int \alpha^{\otimes |E(\gamma)|} d\mathbf{t}=\left((s\otimes\id)\left[\innprod^{-1}\right]\right)^{\otimes |E(\gamma)|}\in A^{\otimes 2|E(\gamma)|},\]
and then apply the map $T$ to it. Hence, the number which is assigned to the graph $\gamma$ is computed by attaching to each edge of $\gamma$ the tensor $(s\otimes\id)\left[\innprod^{-1}\right]$ and contracting these tensors using the map $T$, which applies to each vertex of the stable ribbon graph $\gamma$, the open topological field theory assigned to that vertex by the topological surface which decorates it. This number is exactly that which is assigned to the stable ribbon graph $\gamma$ by the dual construction of Kontsevich, as formulated by Chuang-Lazarev in \cite{dualfeyn}.

\begin{rem}
Since the cocycle in $C^\bullet(\slcmppt)$ defined by Costello's construction agrees with the cocycle defined by Kontsevich's dual construction, it follows  from Proposition 6.1 of \cite{minmod} that its cohomology class does not depend upon the choice of a Hodge decomposition \eqref{eqn_hodge}.
\end{rem}

\begin{rem}
This result differs from Proposition 5.1.1 of \cite{costform}, since the proposition in the cited source concerns the construction described in Theorem 1.2 of \cite{kontfeyn}, whereas the above result concerns the construction described in Theorem 1.3 of the same paper.
\end{rem}

\end{document}